\documentclass [12pt,a4paper,reqno]{amsart}
\input amssymb.sty
\usepackage{amsmath,amsthm}
\usepackage{amsfonts}
\usepackage{datetime}
\allowdisplaybreaks

\newcounter{lemma}[section]

\newcounter{corol}[section]

\newcounter{rem}[section]

\newcounter{theo}[section]

\newcounter{propo}[section]

\begin{document}

\title[Equicontinuity and normality]{Equicontinuity and normality of mappings with integrally bounded $p$-moduli}

\author[A. Golberg, R. Salimov and E. Sevost'yanov]{Anatoly Golberg, Ruslan Salimov and Evgeny Sevost'yanov}

\date{\today \hskip 3mm \currenttime \hskip 4mm (\texttt {ENMIBM.tex})}

\begin{abstract}
We consider the generic discrete open mappings in ${\Bbb
R}^n$ under which the perturbation of extremal lengths of curve
collections is controlled integrally via $\int Q(x)\eta^p(|x-x_0|)
dm(x)$ with $n-1<p<n$, where $Q$ is a measurable function on ${\Bbb
R}^n$ and $\int\limits_{r_1}^{r_2} \eta(r) dr \ge 1$ for any $\eta$
on a given interval $[r_1,r_2].$ We proved that the family of all
open discrete mappings of above type is normal under appropriate restrictions on the majorant $Q.$
\end{abstract}

\maketitle

\bigskip
{\small {\textbf {2010 Mathematics Subject Classification: }
Primary: 30C65, 31B15; Secondary: 30D45}}

\bigskip
{\small {\textbf {Key words:} weighted $p$-module, ring $Q$-mappings, lipschitz mappings,
quasiconformality in the mean, equicontinuity, normal families.}}

\medskip

\bigskip

\section{Introduction}
\subsection{} This paper continues our research of the generic properties of mappings with integrally bounded distortions. We consider discrete open mappings $f:D\to \overline{\mathbb R^n}$, $n\ge 2$, of domains $D\subset \mathbb R^n$, satisfying the inequality controlling the distortion of the distinguished $p$-module ($p>1$) by
\begin{equation}\label{eq3*!}
\mathcal M_p(f(\Gamma(S_1,S_2,A)))\le \int\limits_A Q(x)\,\eta^p(|x-x_0|)\, dm(x),
\end{equation}
when the test subdomains $A\subset D$ are spherical rings
$A=A(r_1,r_2,x_0)=\{\, x\in D\,:\,r_1<|x-x_0|<r_2\}$, $0<r_1<r_2<r_0:={\rm dist}\;(x_0,\partial D)$,
and $\eta$ is arbitrary measurable function $\eta:(r_1,r_2)\to [0,\infty]$ such that
\begin{equation}\label{eq9}
\int\limits_{r_1}^{r_2} \eta(r) dr \ge 1,
\end{equation}
while $Q\,:\,D\to [0,\infty]$ in (\ref{eq3*!}) is a given measurable function. The point $x_0$ is fixed in $D$.

The mappings satisfying (\ref{eq3*!}) are called \textit{ring
$(p,Q)$-mappings at the point $x_0$} (cf. \cite{Sev11}). Note also that
the integral in (\ref{eq3*!}) can be treated as a weighted module
(cf. \cite{AC71}).

Homeomorphisms of such type satisfying a slightly stronger condition than (\ref{eq3*!}) have been introduced in \cite{Gol09}. If $Q\in L^1_{\rm loc}$ and $n-1<p<n$, such homeomorphisms possess various differential properties: ACL, differentiability almost everywhere, Lusin's $(N)$-property, boundedness of Jacobian in the terms of $Q$, etc. All these properties are close to the features of lipschitz homeomorphisms (see, \cite{Gol09}, \cite{GS12}). For discrete open mappings the similar problems have been investigated in \cite{SS11}.

Recall that a mapping $f:D\to \overline{\mathbb R^n}$ is called lipschitz at $x_0\in D$ if there exist a constant $K$, $0<K<\infty$, such that
\begin{equation*}
\limsup\limits_{x\to x_0}\frac{|f(x)-f(x_0)|}{|x-x_0|}\le K.
\end{equation*}

Assuming the uniform boundedness of $Q(x)$ by $K$, one can derive that the mappings satisfying (\ref{eq3*!})-(\ref{eq9}) at $x_0$ are lipschitz at this point (see \cite[Theorem 3]{Geh71}).

On the other hand, from the topological point of view, the origins of the theory of $(p,Q)$-mappings go back to the mappings with bounded distortion (or quasiregular mappings) introduced by Reshetnyak in 1966 (see, e.g. \cite{Resh89}, \cite{Ric93}). A natural extension of quasiregular mappings based on their geometric (modular) description leads to so-called $Q$-mappings, having an independent interest (see \cite{MRSY09}). The main part of the theory of these mappings concerns various interconnections between the properties of the majorant $Q(x)$ and the corresponding properties of the mappings themselves (cf. \cite{BGMV03}, \cite{GRSY12}). Note that the main underlying features of all above mappings are openness and discreteness.

For illustration, we use the following example of ring $(p,Q)$-homeomorphisms for $p=n$
from \cite{GG09}. The quasiconformal homeomorphism
\begin{equation*}
g_1(x)=(x_1\cos\theta-x_2\sin\theta, x_1\sin\theta+x_2\cos\theta, x_3,\ldots,x_n),\quad f(0)=0,\quad |x|<1,
\end{equation*}
where $x=(x_1,\ldots,x_n)$ and $\theta=\log(x_1^2+x_2^2)$, is an automorphism of the unit ball in $\mathbb R^n$.
Its dilatation coefficient equals $(1+\sqrt 2)^n$, so this mapping is a $Q$-ho\-me\-o\-mor\-phism with $Q(x)=(1+\sqrt 2)^n>1$ at all $x\ne 0$.

In the case of a ring $Q$-homeomorphism at the origin, the controlling function $Q$ can be replaced by the angular dilatation. Accordingly, in this example one can put $Q(x)\equiv 1$. On the other hand, any $Q$-ho\-me\-o\-mor\-phism with $Q \equiv 1$ must be, due to the well-known Liouville theorem, a restriction of a M\"obius transformation.

Now we construct a discrete open $(p, Q)$-mapping $g_2:{\mathbb
R}^n\rightarrow {\mathbb R}^n$ with $p\ne n$ and locally integrable $Q$ which is not a local homeomorphism. Consider a rotation around an axis in ${\mathbb R}^n,$ $n\ge 2,$ defined as follows. Let $m\in {\mathbb N},$ $x=(x_1,\ldots, x_{n-2}, x_{n-1},
x_n)\in {\mathbb R}^n.$ For $x_{n-1}^2+x_{n}^2=0$ set $g_2(x)=x,$
and for $x_{n-1}^2+x_{n}^2\ne 0,$ $x_{n-1}=r\cos\varphi,$
$x_{n}=r\sin\varphi$ ($r=\sqrt{x_{n-1}^2+x_n^2},$ $\varphi\in [0,
2\pi)$), let
\begin{equation*}
g_2(x)=(x_1,\ldots, x_{n-2}, r\cos
m\varphi, r\sin m\varphi).
\end{equation*}
For this mapping, $l(g_2^{\,\prime}(x))=1$ and $J(x, g_2)=m$; see \cite[Section~4.3,
Ch.~I]{Resh89}. Then the dilatation coefficient defined by $K_{I, p}(x,
g_2)=J(x,g_2)/l^p(g_2^{\,\prime}(x))$ is identically equal to  $m={\rm const}.$ This is a mapping with bounded distortion by Reshetnyak (or quasiregular). In addition, $g_2$ satisfies the module inequality
\begin{equation*}
\mathcal M_p(g_2(\Gamma))\le
m\, \mathcal M_p(\Gamma)
\end{equation*}
for any family of curves $\Gamma.$ Thus $g_2$ also satisfies
(\ref{eq3*!})--(\ref{eq9}) at any point $x_0\in {\mathbb R}^n.$ Note that $g_2$ is not a homeomorphism in neighborhoods of the points of
the $(n-2)$-dimensional plane $P=\{x\in {\mathbb R}^n:
x=(x_1,\ldots,x_n),\,x_{n-1}^2+x_{n}^2=0\}.$ Composing this mapping with $g_1,$ one obtains a discrete open ring $(p,Q)$-mapping at the origin.

\medskip
In our recent paper \cite{GSS13}, we established that under appropriate conditions on the majorant $Q$, the mappings obeying (\ref{eq3*!}) may have only removable singularities, i.e. neither poles nor essential singularities.

\subsection{} The purpose of this paper is to investigate the normality of families of discrete open continuous mappings which are subject
to (\ref{eq3*!}). The normal families play an important role in Complex Analysis. In fact, they are pre-compact families of continuous functions. Roughly speaking, the functions in the family provide a somewhat ``closeness'' and behave in a relatively ``compact'' manner. The notion of normality was introduced about 100 years ago and became a crucial tool in studying compact sets in the spaces of functions, which are infinite-dimensional in nature (see, e.g. \cite{Zal98}).
More precisely, a family $\mathcal F$ of continuous functions $f$ defined on some metric space $(X, d_X)$ with values in another metric space $(Y, d_Y)$ is called \textit{normal} if every sequence of functions in $\mathcal F$ contains a subsequence which converges uniformly on compact subsets of $X$ to a continuous function from $X$ to $Y$.
Another notion closely related to normality is the equicontinuity of a function family. The family $\mathcal F$ is equicontinuous at a point $x_0\in X$ if for every $\varepsilon > 0$, there exists $\delta > 0$ such that $d_Y(f(x_0), f(x)) < \varepsilon$ for all $f\in \mathcal F$ and all $x$ such that $d_X(x_0, x) < \delta$. The family is equicontinuous if it is equicontinuous at each point of $X$. Thus, by the well-known Ascoli theorem, normality is equivalent to equicontinuity on compact sets of the functions in $\mathcal F$.

\medskip
It was established in \cite{Sev07}, that if a family of discrete open mappings
$f:D\setminus\{x_0\}\rightarrow \overline{{\mathbb R}^n},$ $n\ge 2,$ which omits the values ranging over some set of positive
conformal capacity and obeys (\ref{eq3*!}) for $p=n$, is normal (cf. \cite{Cris10}). In this paper we are focused to the case $n-1<p<n.$
The range $(n-1,n)$ for $p$ is needed in order to apply, similar to the classical case $p=n$, the ``continuum'' analysis, since for $n-1<p<n$ there appear some essential differences from the case $p=n$. One of those is the role of the infinite point $x=\infty.$ Namely, if $p=n,$ the point $\infty$ plays the same role as any other point in $\mathbb R^n$. In contract, when $p < n,$ the point $\infty$ has positive $p$-capacity (cf. \cite[Lemma 2.3]{GSS13}).  Another essential difference allows us to establish normality without an expected requirement ${\rm cap}_p\left({\mathbb R}^n\setminus
f(D\setminus \{x_0\})\right)>0.$

\medskip
\subsection{} A function ${\varphi}:D\rightarrow{\mathbb R}$ is of {\it finite mean oscillation} at a point $x_0\in D,$ $\varphi\in{\rm FMO}(x_0),$  if
\begin{equation*}
\overline{\lim\limits_{\varepsilon\rightarrow 0}}
\,\,\frac{1}{\Omega_n\varepsilon^n}\int\limits_{B( x_0,\,\varepsilon)}
|{\varphi}(x)-\overline{{\varphi}}_{\varepsilon}| dm(x)<\infty,
\end{equation*}
where
\begin{equation*}
\overline{{\varphi}}_{\varepsilon}=\frac{1}{\Omega_n\varepsilon^n}\int\limits_{B(
x_0,\,\varepsilon)} {\varphi}(x)\ dm(x),
\end{equation*}
and $\Omega_n$ denotes the volume of the unit ball  ${\mathbb B}^n$ in ${\mathbb R}^n.$  The class ${\rm FMO}$ was introduced in \cite{IR05} and provides a natural extension of the well-known class of mappings with bounded mean oscillation (${\rm BMO}$) widely applied in Harmonic Analysis, PDE and many other areas.

\medskip
The following statement is one of the main results of the paper. We make the convention that henceforth the term ``$f$ is discrete open'' includes the assumption that $f$ is also continuous.

\medskip
\begin{theo}\label{th1}
{\em Let $Q:D\rightarrow[0, \infty]$ be a Lebesgue measurable
function, $b\in D,$ $n-1<p<n,$ and let $\frak{F}_{p, Q}(b)$ be a
family of all open discrete $(p, Q)$-mappings $f:D\rightarrow {\mathbb
R}^n$ at $b.$ Assume that $Q\in FMO(b).$ Then $\frak{F}_{p, Q}(b)$ is equicontinuous at $b.$ Moreover, if
$\frak{F}_{p, Q}(D)$ is a family of all open discrete $(p,
Q)$-mappings $f:D\rightarrow {\mathbb R}^n$ at every point $b\in D$ and
the condition $Q\in {\rm FMO}(b)$ holds for every point $b\in D,$ then the
family $\frak{F}_{p, Q}(D)$ is normal in $D.$ }
\end{theo}

\medskip
As was mentioned above, the assertion of Theorem \ref{th1} is not
true for $p=n$. As an example, one can take $f_m(z)=e^{mz},$ $z\in {\mathbb C},$ $b=0$.

\section{Auxiliary results}

\subsection{} We start with introducing $p$-capacity and $p$-module of arbitrary sets and describe their basic features, lower growth estimates, etc.

\medskip
Recall that a pair
$E=\left(A,\,C\right),$ where $A$ is an open set in ${\mathbb R}^n,$
and $C$ is a compact subset of $A,$ is called \textit{condenser} in ${\mathbb R}^n$.  The quantity
\begin{equation}\label{eq4}{\rm cap}_p\,E\quad=\quad{\rm
cap}_p\,\left(A,\,C\right)\quad=\quad\inf\limits_{u\,\in\,W_0\left(E\right)
}\quad\int\limits_{A}\,|\nabla u|^p\,\,dm(x)\,,
\end{equation}
where $W_0(E)=W_0\left(A,\,C\right)$ is a family of all nonnegative
absolutely continuous on lines (ACL) functions $u:A\rightarrow {\mathbb R}$ with compact support
in $A$ and such that $u(x)\ge 1$ on $C,$ is called the \textit{$p$-capacity} of
condenser $E$.

We say that a compact set $C$ in ${\mathbb R}^n,$ $n\ge
2,$ has \textit{zero $p$-capacity} (and write ${\rm cap}_p\,C=0$), if there exists
a bounded open set $A$ such that $C\subset A$ and ${\rm cap}_p\,(A,
C)=0.$ Otherwise, the set $C$ has positive $p$-capacity, ${\rm cap}_p\, C>0.$
An arbitrary set $F\subset{\mathbb R}^n$ is of $p$-capacity zero if all its compact subsets have zero capacity.

\medskip
The basic lower estimate of $p$-capacity of a condenser $E=(A, C)$ in
${\mathbb R}^n$ is given by
\begin{equation}\label{2.5} {\rm cap}_p\ E = {\rm cap}_p\ (A, C) \ge
\left(c_1\frac{\left(d(C)\right)^p}
{\left(m(A)\right)^{1-n+p}}\right)^{\frac{1}{n-1}}\,, \qquad p>n-1,
\end{equation}
where $c_1$ depends only on $n$ and $p$ and $d(C)$ denotes the diameter of $C$ (see
\cite[Proposition~6]{Kru86}).

\medskip
The following proposition provides the lower estimate for $p$-capacities involving the compactum volume; see e.g. \cite[(8.9)]{Maz03}.

\medskip
\begin{propo}\label{pr1A} 
{\em Given a condenser $E=(A, C)$ and  $1<p<n,$
\begin{equation*}
{\rm cap}_p\,E\ge n{\Omega}^{\frac{p}{n}}_n
\left(\frac{n-p}{p-1}\right)^{p-1}\left[m(C)\right]^{\frac{n-p}{n}}\,,
\end{equation*}
where ${\Omega}_n$ denotes the volume of the unit ball in ${\Bbb
R}^n$, and $m(C)$ is the $n$-dimensional Lebesgue measure of $C$.}
\end{propo}

\medskip
A curve $\gamma$ in ${\mathbb R}^n$ $(\,\overline{{\mathbb R}^n}\,)$ is a continuous mapping $\gamma
:\Delta\rightarrow{\mathbb R}^n\;(\,\overline{{\mathbb R}^n}\,),$ where $\Delta$ is an interval in
${\mathbb R} .$ Its locus $\gamma(\Delta)$ is denoted by $|\gamma|.$
Let $\Gamma$ be a family of curves $\gamma$ in ${\mathbb R}^n .$ A Borel
function $\rho:{\mathbb R}^n \rightarrow [0,\infty]$ is called {\it
admissible} for $\Gamma ,$ abbr. $\rho \in {\rm adm}\, \Gamma ,$ if
$\int\limits_{\gamma} \rho(x)|dx| \ge 1$
for each (locally rectifiable) $\gamma\in\Gamma.$ For $p\ge 1,$
we define the quantity
\begin{equation*}
\mathcal M_p(\Gamma)\,=\,\inf\limits_{ \rho \in {\rm adm}\, \Gamma}
\int\limits_{{\mathbb R}^n} \rho^p(x) dm(x)
\end{equation*}
and call $\mathcal M_p(\Gamma)$ {\it $p$-mo\-du\-le} of $\Gamma$ \cite[6.1]{Vai71}; here $m$ stands for the $n$-dimensional Lebesque measure (see also \cite{Vuo88}, \cite{Ric93}).

If ${\rm adm}\, \Gamma = \varnothing $ we set $\mathcal M_p(\Gamma)=\infty$. Note that
$\mathcal M_p(\varnothing)=0;$ $\mathcal M_p(\Gamma_1)\le \mathcal M_p(\Gamma_2)$ whenever
$\Gamma_1\subset\Gamma_2,$ and moreover
$\mathcal M_p\left(\bigcup\limits_{i=1}^{\infty}\Gamma_i\right)\le
\sum\limits_{i=1}^{\infty}\mathcal M_p(\Gamma_i)$ (see
\cite[Theorem~6.2]{Vai71}).

We say that $\Gamma_1$ is \textit{minorized} by $\Gamma_2$ and write $\Gamma_2 < \Gamma_1$ if every $\gamma\in\Gamma_1$ has a subcurve which belongs to $\Gamma_2$. The relation $\Gamma_2 < \Gamma_1$ implies that ${\rm adm}\,\Gamma_2 \subset {\rm adm}\,\Gamma_1$ and therefore $\mathcal M_p(\Gamma_1)\le \mathcal M_p(\Gamma_2)$.

\medskip
Let $E_0,$ $E_1$ are two sets in $D\subset \overline{\mathbb R^n}$. Denote by $\Gamma (E_0, E_1, D)$ a family of al curves joining $E_0$ and $E_1$ in $D$. For such families, we have the following estimate  (see \cite[Theorem~4]{Car94}).

\medskip
\begin{propo}\label{pr1} 
{\em Let $A(a,
b, 0)=\{a<|x|<b\}$ be a ring containing in $D$ such that for every $r\in (a,b)$ the sphere $S(0,
r)$ meets both $E_0$ and $E_1,$ where $E_0\cap
E_1=\varnothing.$ Then for every $p\in(n-1, n),$
\begin{equation*}
\mathcal M_p\left(\Gamma\left(E_0, E_1, D\right)\right)\ge
\frac{2^nb_{n,p}}{n-p}\left(b^{n-p}-a^{n-p}\right)\,,
\end{equation*}
where $b_{n,p}$ is a constant depending only on $n$ and $p$.}
\end{propo}

\medskip
We also need the following statement given in \cite[Proposition~10.2,
Ch.~II]{Ric93}.

\medskip
\begin{propo}\label{lem2.2} 
{\em\, Let $E=(A,\,C)$ be a condenser in ${\mathbb R}^n$ and let
$\Gamma_E$ be the family of all curves of the form
$\gamma:[a,\,b)\rightarrow A$ with $\gamma(a)\in C$ and
$|\gamma|\cap\left(A\setminus F\right)\ne\varnothing$ for every
compact $F\subset A.$ Then ${\rm cap}_q\,E=\mathcal
M_q\left(\Gamma_E\right).$}
\end{propo}

\medskip
Note that Proposition \ref{lem2.2} allows us to give a natural extension  of $p$-capacity of a condenser $E\subset\overline{\mathbb R^n}$ by ${\rm cap}_q\,E=\mathcal
M_q\left(\Gamma_E\right).$

\subsection{}
Let $f:D \rightarrow {\mathbb R}^n$ be a discrete open mapping, $\beta:
[a,\,b)\rightarrow {\mathbb R}^n$ be  a curve, and
$x\in\,f^{-1}\left(\beta(a)\right).$ A curve $\alpha:
[a,\,c)\rightarrow D$ is called a {\it maximal $f$-lifting} of
$\beta$ starting at $x,$ if $(1)\quad \alpha(a)=x\,;$ $(2)\quad
f\circ\alpha=\beta|_{[a,\,c)};$ $(3)$\quad for $c<c^{\prime}\le b,$
there is no curves $\alpha^{\prime}: [a,\,c^{\prime})\rightarrow
D$ such that $\alpha=\alpha^{\prime}|_{[a,\,c)}$ and $f\circ
\alpha^{\,\prime}=\beta|_{[a,\,c^{\prime})}.$  The assumption on $f$
yields that every  curve $\beta$ with $x\in
f^{\,-1}\left(\beta(a)\right)$ has  a maximal $f$-lif\-ting starting
at $x$ (see \cite[Corollary~II.3.3]{Ric93}, \cite[Lemma~3.12]{MRV71}).

\begin{lemma}\label{lem4} 
{\em Let $f:D\rightarrow\overline{{\mathbb R}^n},$ $n\ge 2,$ be an open
discrete ring $(p, Q)$-map\-ping at a point $x_0\in D.$ Suppose that
there exist numbers $\varepsilon_0 \in (0,\, {\rm dist}\,(x_0,
\partial D)),$ $\varepsilon_0^{\,\prime}\in (0, \varepsilon_0)$ and
a family of nonnegative Lebesgue measurable functions
$\{\psi_{\varepsilon}(t)\},$ $\psi_{\varepsilon}:(\varepsilon,
\varepsilon_0)\rightarrow [0, \infty],$ $\varepsilon\in\left(0,
\varepsilon_0^{\,\prime}\right),$ such that
\begin{equation}\label{eq3.7B}
\int\limits_{\varepsilon<|x-x_0|<\varepsilon_0}Q(x)\psi_{\varepsilon}^p(|x-x_0|)
\ dm(x)\le \Phi(\varepsilon,
\varepsilon_0)\quad\text{for all}\quad \varepsilon\in(0,
\varepsilon_0^{\,\prime})\,,
\end{equation}
where $\Phi(\varepsilon, \varepsilon_0)$ is a given bounded function and
\begin{equation}\label{eq3}
0<I(\varepsilon, \varepsilon_0):=
\int\limits_{\varepsilon}^{\varepsilon_0}\psi_{\varepsilon}(t)dt <
\infty\quad\text{for all}\quad\varepsilon\in(0,
\varepsilon_0^{\,\prime})\,.
\end{equation}
Then for the condensers $E=\left(A,\,C\right)$ with $A=B\left(x_0, r_0\right),$
$C=\overline{B(x_0, \varepsilon)},$ $r_0={\rm
dist}\left(x_0,\,\partial D\right)$  ($A={\mathbb R}^n$ whenever
$D={\mathbb R}^n$)
\begin{equation}\label{eq3B}
{\rm cap}_p\,f(E)\le \Phi(\varepsilon,\varepsilon_0)/
I^{p}(\varepsilon, \varepsilon_0)\quad\text{for all}\quad\varepsilon\in\left(0,\,\varepsilon_0^{\,\prime}\right)\,.
\end{equation}}
\end{lemma}

\begin{proof} Let $E=(A,\,C)$ a condenser  with $A$ and $C$
defined as above. By the continuity and openness of $f,$ the pair
$f(E)=(f(A), f(C))$ is also a condenser.
In the case ${\rm cap}_p\,f(E)=0$ the inequality (\ref{eq3B}) is trivial; thus, one can assume that
${\rm cap}_p\,f(E)\ne 0.$

\medskip
Let $\Gamma_E$ be a family of curves $\gamma:[a,\,b)\rightarrow A$ such that $\gamma(a)\in C$ and
$|\gamma|\cap\left(A\setminus F\right)\ne\varnothing$ for every
compact set $F\subset A$, where $|\gamma|=\{x\in {\mathbb R}^n:
\exists\,t\in [a, b): \gamma(t)=x \}$ is a locus of the curve
$\gamma.$ By Proposition \ref{lem2.2}, ${\rm cap}_p\,E=\mathcal M_p(\Gamma_E)$.

\medskip
Consider a similar family of curves $\Gamma_{f(E)}$ for the condenser
$f(E)$ and observe that every $\gamma\in\Gamma_{f(E)}$ has a maximal
$f$-lif\-ting in $A$ starting in $C$ (cf. \cite[Corollary 3.3,
Ch.~II]{Ric93}). Let $\Gamma^{*}$ be a family of all maximal
$f$--lif\-tings of $\Gamma_{f(E)}$ starting in $C.$ Note that
$\Gamma^{*}\subset \Gamma_E.$ Then
$\Gamma_{f(E)}>f(\Gamma^{*}),$ and hence,
\begin{equation}\label{eq7}
\mathcal M_p\left(\Gamma_{f(E)}\right)\le \mathcal M_p\left(f(\Gamma^{*})\right)\,.
\end{equation}
Consider two spheres $S_{\,\varepsilon}=S(x_0,\,\varepsilon),$ $
S_{\,\varepsilon_0}=S(x_0,\,\varepsilon_0),$ bounding the ring
$A(\varepsilon, \varepsilon_0, x_0)=\{x\in {\mathbb R}^n: \varepsilon<|x-x_0|<\varepsilon_0\},$
where $\varepsilon_0$ is given in assumption of lemma, and
$\varepsilon\in\left(0,\,\varepsilon_0^{\,\prime}\right).$
Since $\Gamma\left(S_{\varepsilon}, S_{\varepsilon_0},
A(\varepsilon, \varepsilon_0, x_0)\right)<\Gamma^{\,*}$ and
hence, $f(\Gamma\left(S_{\varepsilon}, S_{\varepsilon_0},
A(\varepsilon, \varepsilon_0, x_0)\right))<f(\Gamma^{*}),$ we have
\begin{equation}\label{eq5}
\mathcal M_p\left(f(\Gamma^{*})\right)\le
\mathcal M_p\left(f\left(\Gamma\left(S_{\varepsilon}, S_{\varepsilon_0},
 A(\varepsilon, \varepsilon_0, x_0)\right)\right)\right)\,.
\end{equation}
From (\ref{eq7}) and (\ref{eq5}),
\begin{equation*}
\mathcal M_p\left(\Gamma_{f(E)}\right)\,\le
\mathcal M_p\left(f\left(\Gamma\left(S_{\varepsilon}, S_{\varepsilon_0},
 A(\varepsilon, \varepsilon_0, x_0)\right)\right)\right).
\end{equation*}
Applying Proposition \ref{lem2.2}, one gets
\begin{equation}\label{eq5aa}
{\rm cap}_p\,\,f(E) \le
\mathcal M_p\left(f\left(\Gamma\left(S_{\varepsilon}\,, S_{\varepsilon_0},
 A(\varepsilon, \varepsilon_0, x_0)\right)\right)\right)\,.
\end{equation}

Now pick a family of Lebesgue measurable functions
$\eta_{\varepsilon}(t)=\psi_{\varepsilon}(t)/I(\varepsilon,
\varepsilon_0 ),$ $t\in(\varepsilon,\, \varepsilon_0).$ For any
$\varepsilon\in (0, \varepsilon_0^{\,\prime}),$
$\int\limits_{\varepsilon}^{\varepsilon_0}\eta_{\varepsilon}(t)\,dt=1,$
and by the definition of ring $Q$-mappings at the point $x_0,$
\begin{equation}\label{eq8***}
\mathcal M_p\left(f\left(\Gamma\left(S_{\varepsilon}\,, S_{\varepsilon_0},
A(\varepsilon, \varepsilon_0, x_0)\right)\right)\right)\,
\le\frac{1}{I^p(\varepsilon, \varepsilon_0)}
\int\limits_{\varepsilon<|x-x_0|<\varepsilon_0}Q(x)\psi_{\varepsilon}^p(|x-x_0|)\,
\ dm(x)
\end{equation}
for any $\varepsilon\in (0, \varepsilon_0^{\,\prime}).$ Now the estimate
(\ref{eq3B}) directly follows from (\ref{eq3.7B}),
(\ref{eq5aa}) and (\ref{eq8***}).
\end{proof}

\medskip
\subsection{} The chordal metric $h$ in $\overline{{\mathbb
R}^n}$ is defined by
\begin{equation*}
h(x,\infty)=\frac{1}{\sqrt{1+{|x|}^2}}, \ \
h(x,y)=\frac{|x-y|}{\sqrt{1+{|x|}^2} \sqrt{1+{|y|}^2}}\,, \ \  x\ne
\infty\ne y\,.
\end{equation*}
The {\it chordal diameter} of a set $E\,\subset\,\overline{{\mathbb
R}^n}$ is defined by (see, e.g. \cite{Vuo88})
\begin{equation}\label{chdiam}
h(E)=\sup\limits_{x\,,y\,\in\,E}\,h(x,y)\,.
\end{equation}

The following statement was proved for $p=n$ in
\cite[Lemma~3.11]{MRV$_2$} (see also \cite[Lemma~2.6,
Ch.~III]{Ric93}).

\medskip
\begin{lemma}\label{lem1A} 
{\em\, Let $F$ be a compact proper subset of $\overline{{\mathbb R}^n}$
with ${\rm cap}_p\,\left(F\setminus\{\infty\}\right)>0,$ $n-1<p<n.$ Then for every $a>0$ there
exists $\delta>0$ such that
\begin{equation*}
{\rm cap}_p\,\left(\overline{{\mathbb
R}^n}\setminus F,\, C\right)\ge \delta
\end{equation*}
for every
continuum $C\subset \overline{{\mathbb R}^n}\setminus F$ with $h(C)\ge
a.$ }
\end{lemma}

\begin{proof}
One can assume that $F$ does not contain $\infty$. Otherwise,
one can find a compact subset $F_1$ of $F$ such that $F_1$ is bounded in $\mathbb R^n$ and ${\rm cap}_p\,F_1>0.$
Passing it needed to $F_1$, we may assume that $F$ is bounded, hence contained in a ball $B(0,R)$ whose radius $R$ can be chosen to be so large that
\begin{equation}\label{eq6A}
h\left(\overline{{\mathbb R}^n}\setminus B(0, R)\right)< a/2\,.
\end{equation}

Now we show that any continuum $C$ with $h(C)\ge a$ (see (\ref{chdiam})) contains a subcontinuum $C_1\subset B(0, R)$ with $h(C_1)\ge a/4.$

If $C\subset B(0, R),$ we set $C_1:=C.$ Otherwise,
$C\cap \left(\overline{{\mathbb R}^n}\setminus B(0, R)\right)\ne \varnothing.$
Since $C$ is closed, one can fix a pair $x_0, y_0\in \overline{{\mathbb R}^n}$
on which the chordal diameter $h(C)$ is attained $(h(C)=h(x_0, y_0)).$
Note that $x_0$ and $y_0$ do not simultaneously belong to the complement of $B(0, R)$, because due to (\ref{eq6A}),
$h\left(\overline{{\mathbb R}^n}\setminus B(0, R)\right)< a/2$ and $h(C)\ge a$ as well.

Without loss of generality one can assume that $x_0\in B(0, R).$ There are two possibilities:

\medskip
1) $y_0\in \overline{{\mathbb R}^n}\setminus B(0, R).$ Let $C_2$ be a
connected component of $C\cap \overline{B(0, R)},$ containing $x_0.$
Since $C$ is connected, there exists $z_1\in C_2\cap S(0, R).$ Then
by the triangle inequality
\begin{equation*}
a\le h(x_0, y_0)\le h(x_0, z_1)+h(z_1, y_0)< h(C_2)+a/2\,,
\end{equation*}
and moreover, $h(C_2)> a/2.$ This component $C_2$ can be chosen as the needed set $C_1$.

\medskip
2) $y_0\in B(0, R).$ Let $C_2$ be as above, and
denote by $C_3$ a connected component of $C\cap \overline{B(0, R)},$
containing $y_0.$ By the triangle inequality, for any point $z_2\in
C_3\cap S(0, R),$
\begin{equation*}
a\le h(x_0, y_0)\le h(x_0, z_1)+h(z_1, z_2)+ h(z_2, y_0)< h(C_2) + h(C_3)+ a/2\,.
\end{equation*}
It follows that either $h(C_2)>a/4$ or $h(C_3)> a/4.$
The corresponding $C_k,$ $k=2,3,$ with $h(C_k)> a/4$ is chosen as the prescribed set $C_1$.
So, in either case there exists a continuum $C_1\subset B(0, R)$ with
$h(C_1)\ge a/4.$

Note that by the definition of $p$-capacity given in
(\ref{eq4}), ${\rm cap}_p\,\left(\overline{{\mathbb R}^n}\setminus F,\,
C\right)\ge {\rm cap}_p\,\left(\overline{{\mathbb R}^n}\setminus F,\,
C_1\right);$ so it suffices to estimate the capacity in the right-hand side.

Arguing similarly to \cite{Ric93}, pick $\Gamma_1:=\Gamma\left(F, S(0, 2R), B(0, 2R)\right).$
Since ${\rm cap}_p\,\left(F\setminus\{\infty\}\right)>0$, one derives from Proposition \ref{lem2.2}
that
\begin{equation}\label{eq7A}
\mathcal M_p(\Gamma_1):=\delta_1>0\,,
\end{equation}
with a constant $\delta_1$
depending only on $p,$ $R$ and  $F.$

Now consider a curve family $\Gamma_2=\left(C_1, S(0, 2R), B(0, 2R)\right).$
Its $p$-module $\mathcal M_p(\Gamma_2)$ can be estimated using the equality
${\rm cap}_p\left(B(0, 2R),
C_1\right)=\mathcal M_p(\Gamma_2)$ from Proposition \ref{lem2.2} and the lower bound (\ref{2.5}). This results in
\begin{equation}\label{eq8}
\mathcal M_p(\Gamma_2) \ge \left(c_1\frac{\left(d(C_1)\right)^p}
{\left(2^n\Omega_n R^n\right)^{1-n+p}}\right)^{\frac{1}{n-1}} \ge
\left(c_1\frac{\left(a/4\right)^p} {\left(2^n\Omega_n
R^n\right)^{1-n+p}}\right)^{\frac{1}{n-1}}:=\delta_2\,,
\end{equation}
where $\delta_2$ depends only on $p,$ $R,$ $a$ and $n.$

For
$\Gamma_{1,2}=\Gamma\left(C_1, F, \overline{{\mathbb R}^n}\right),$
Proposition \ref{lem2.2} yields
\begin{equation}\label{eq9AA}
\mathcal M_p(\Gamma_{1,2})= {\rm cap}_p\,\left(\overline{{\mathbb
R}^n}\setminus F,\, C_1\right).
\end{equation}
Pick arbitrary $\rho\in {\rm
adm\,}\Gamma_{1,2}.$ If either $3\rho\in {\rm adm\,}\Gamma_{1}$ or
$3\rho\in  {\rm adm\,}\Gamma_{2},$ then by the relations (\ref{eq7A})
and (\ref{eq8}),
\begin{equation}\label{eq10A}
\int\limits_{{\mathbb R}^n}\rho^p(x)dm(x)\ge 3^{-p}\min\{\delta_1,
\delta_2\}\,.
\end{equation}

If for both $\Gamma_1$ and $\Gamma_2$ we have $3\rho\not\in  {\rm adm\,}\Gamma_{1}$
and  $3\rho\not\in {\rm adm\,}\Gamma_{2},$ then there exists a pair
of curves $\gamma_1\in \Gamma_1$ and $\gamma_2\in \Gamma_2$ for which
\begin{equation}\label{eq11}
\int\limits_{\gamma_1}\rho (x)\ \ |dx|< 1/3,\qquad
\int\limits_{\gamma_2}\rho (x)\ \ |dx|< 1/3\,.
\end{equation}

Now consider a family of curves $\Gamma_4=\Gamma\left(|\gamma_1|,
|\gamma_2|, B(0, 2R)\setminus \overline{B(0, R)}\right),$ where
$|\gamma|$ denotes a locus of curve $\gamma,$ i.e.,
$|\gamma|=\{x\in {\mathbb R}^n:\, \exists\, t: x=\gamma(t)\}.$ Since
$\rho\in {\rm adm\,}\Gamma_{1,2},$ the relations (\ref{eq11}) give that
$\int\limits_{\gamma}\rho (x)|dx|\ge 1/3$ for any
curve $\gamma\in \Gamma_4.$ In this case, $3\rho\in {\rm
adm\,}\Gamma_{4},$ hence by Proposition \ref{pr1},
\begin{equation}\label{eq13A}
\int\limits_{{\mathbb R}^n}\rho^p(x)dm(x)\ge 3^{-p}\mathcal
M_p(\Gamma_4)\ge \delta_3\,,
\end{equation}
where $\delta_3$ depends only on $n,$ $p$ and $R.$ Combining (\ref{eq10A})
and (\ref{eq13A}), one gets the estimate
\begin{equation*}
\mathcal M_p(\Gamma_{1,2})\ge \min\{\delta_1, \delta_2,
\delta_3\}:=\delta\,.
\end{equation*}
Replacing $p$-module by $p$-capacity as in (\ref{eq9AA}), one completes the proof of Lemma \ref{lem1A}.
\end{proof}

\subsection{}
The following statement provides the integral conditions for the majorant $Q$ and admissible metric $\psi,$ which ensure the normality for $(p,Q)$-mapping families omitting the sets of positive $p$-capacity.

\medskip
\begin{lemma}\label{lem3*} 
{\em Let $E \subset\,\overline{{\mathbb R}^n}$ be a compact set with ${\rm cap}_p\,\left(E\setminus\{\infty\}\right)>0,$ $n-1<p<n,$ and
${\frak F_{Q}}$ be a family of open discrete ring $(p,
Q)$-mappings $f:D\rightarrow\overline{{\mathbb R}^n}\setminus E$ at
$x_0\in D.$  Suppose that
\begin{equation} \label{eq27*!}
\int\limits_{\varepsilon<|x-x_0|<\varepsilon_0}Q(x)\psi_{\varepsilon}^p(|x-x_0|)
\ dm(x)=o\left(I^p(\varepsilon, \varepsilon_0)\right)
\end{equation}
as $\varepsilon\rightarrow 0$ and some $0< \varepsilon_0 < {\rm
dist\,}(x_0,
\partial D),$ where $\{\psi_{\varepsilon}(t)\}$ is a family of
nonnegative Lebesgue measurable functions on $(0,\varepsilon_0)$
satisfying
\begin{equation}\label{eq1A}
0<I(\varepsilon,\varepsilon_0)=\int\limits_{\varepsilon}^{\varepsilon_0}\psi_{\varepsilon}(t)dt
< \infty,\qquad\forall\ \ \varepsilon \in(0,
\varepsilon^{\,\prime}_0)\,, \quad 0<\varepsilon^{\,\prime}_0<\varepsilon_0.
\end{equation}
Then the family ${\frak F_{Q}}$ is equicontinuous at $x_0.$}
\end{lemma}

\begin{proof} Consider a condenser ${\mathcal
E}=(A,\,C)$ with $A=B\left(x_0,\,r_0\right)$ and
$C=\overline{B(x_0,\,\varepsilon)},$ where $r_0\,=\,{\rm
dist\,}\left(x_0,\,\partial D\right).$ Assume that $D={\mathbb R}^n,$ i.e. $r_0=\infty.$ Then, by Lemma \ref{lem1A}, for any positive $a>0$ there exists $\delta =\delta(a)$ such that
${\rm cap}_p\,\left(\overline{{\mathbb
R}^n}\setminus E,\, C\right)\ge \delta$. On the other hand,
the estimate (\ref{eq3B}) from Lemma
\ref{lem4} and the asymptotic relation (\ref{eq27*!}) yield
\begin{equation*}
{\rm cap\,}_p\,f\left({\mathcal E}\right)\le \alpha(\varepsilon)\quad\text{for
any}\quad \varepsilon\,\in (0,\,\varepsilon^{\,\prime}_0),
\end{equation*}
with $\alpha(\varepsilon)\,\rightarrow\,0$ as
$\varepsilon\,\rightarrow 0.$ Then, for $\delta=\delta(a)$
there exists $\varepsilon_*=\varepsilon_*(a)$ such that if $\varepsilon\,\in \left(0,\,{\varepsilon_*}(a)\right),$
one gets
\begin{equation}\label{eq28*!}
{\rm cap\,}_p\,f\left({\mathcal E}\right)\le \delta\,.
\end{equation}
By (\ref{eq28*!}),
\begin{equation*}
{\rm cap\,}_p\,\left(\overline{{\mathbb R}^n}\setminus
E,\,f\left(\overline{B(x_0,\,\varepsilon)}\right)\right)\,\le\, {\rm
cap\,}_p\,\left(f\left({B(x_0,\,r_0)}\right)\,,\,
f\left(\overline{B(x_0,\,\varepsilon)}\right)\right)\le\,\delta
\end{equation*}
as $\varepsilon\,\in \left(0,\,{\varepsilon_*}(a)\right).$ Thus, by
 Lemma \ref{lem1A},
$h\left(f\left(\overline{B(x_0,\,\varepsilon)}\right)\right)\,<\,a.$
Finally, for any $a\,>\,0$ there exists
$\varepsilon_*\,=\,\varepsilon_*(a)$ with
$h\left(f\left(\overline{B(x_0,\,\varepsilon)}\right)\right)\,<\,a\,$
provided that $\varepsilon\,\in \left(0,\,{\varepsilon_*}(a)\right).$ This completes
the proof of the lemma.
\end{proof}

\medskip
The following lemma is a stronger statement on normality of
families of open discrete ring $(p,Q)$-mappings. It
shows, in particular, that the assumption of omitting a set of positive
$p$-capacity in Lemma \ref{lem3*} can be dropped.

\medskip
\begin{lemma}\label{lem1} 
{\em Let $Q:D\rightarrow (0, \infty]$ be
a Lebesgue measurable function, and ${\frak F_{Q}}$ be a family of open
discrete ring $(p, Q)$-mappings $f:D\rightarrow{\mathbb R}^n$ at
$x_0\in D,$ $p\in (n-1, n).$ Suppose that the growth condition
(\ref{eq27*!}) holds as $\varepsilon\rightarrow 0$ and some $0<
\varepsilon_0 < {\rm dist\,}(x_0,
\partial D),$ where $\{\psi_{\varepsilon}(t)\}$ is a family of
nonnegative Lebesgue measurable functions on $(0,\varepsilon_0)$
satisfying (\ref{eq1A}).
Then the family ${\frak F_{Q}}$ is equicontinuous at $x_0.$}
\end{lemma}

\medskip
\begin{proof} Consider a condenser ${\mathcal
E}=(A,\,C)$ with $A=B\left(x_0,\,r_0\right),$
$C=\overline{B(x_0,\,\varepsilon)},$ where, as usual, $r_0\,=\,{\rm
dist\,}\left(x_0,\,\partial D\right).$ If $D={\mathbb R}^n,$ i.e. $r_0=\infty,$
then the estimate (\ref{eq3B}) and the asymptotic relation (\ref{eq27*!}) again yield
${\rm cap\,}_p\,f\left({\mathcal E}\right)\le \alpha(\varepsilon)$ for
any $\varepsilon\,\in (0,\,\varepsilon^{\,\prime}_0),$
with $\alpha(\varepsilon)\,\rightarrow\,0$ as
$\varepsilon\,\rightarrow 0.$ Applying Proposition \ref{pr1A}, one
obtains
\begin{equation*}
\alpha(\varepsilon)\ge{\rm cap}_p\,f(E)\ge n{\Omega}^{\frac{p}{n}}_n
\left(\frac{n-p}{p-1}\right)^{p-1}\left[m(f(C))\right]^{\frac{n-p}{n}}\,,
\end{equation*}
where ${\Omega}_n$ denotes the volume of the unit ball in ${\mathbb
R}^n,$ and $m(C)$ stands for the $n$-dimensional Lebesgue measure of $C.$
In other words,
\begin{equation*}   
m(f(C))\le \alpha_1(\varepsilon)\,,
\end{equation*}
where $\alpha_1(\varepsilon)\rightarrow 0$ as $\varepsilon\rightarrow
0.$ Choosing appropriate $\varepsilon_1\in
(0, 1),$ one has for $\varepsilon<\varepsilon_1$ a more rough bound
\begin{equation}\label{eqroughb}
m(f(C))\le 1,
\end{equation}
where $C=\overline{B(x_0, \varepsilon_1)}.$

Applying the inequalities
(\ref{eq3B}) and (\ref{eq27*!}) to the condenser
${\mathcal E}_1=(A_1,\,C_{\varepsilon}),$
$A_1=B\left(x_0,\,\varepsilon_0\right),$ and
$C_{\varepsilon}=\overline{B(x_0, \varepsilon)},$ $\varepsilon\in
(0, \varepsilon_1),$  yields
\begin{equation*}   
{\rm cap\,}_p\,f\left({\mathcal E}_1\right)\le \alpha_2(\varepsilon)
\end{equation*}
for all $\varepsilon\in (0,\,\varepsilon^{\,\prime}_0),$
where $\alpha_2(\varepsilon)\rightarrow0$ as $\varepsilon\rightarrow
0,$ and after estimating $p$-capacity from below in the left-hand side by
(\ref{2.5}),
\begin{equation*}
\left(c_1\frac{\left(d(f(\overline{B(x_0, \varepsilon)}))\right)^p}
{\left(m(f(B(x_0,
\varepsilon_1)))\right)^{1-n+p}}\right)^{\frac{1}{n-1}} \le {\rm
cap\,}_p\,f\left({\mathcal E}_1\right)\le \alpha_2(\varepsilon)\,.
\end{equation*}
Combining this with (\ref{eqroughb}), one gets
\begin{equation*}
d(f(\overline{B(x_0, \varepsilon)}))
\le \alpha_3(\varepsilon)\,,
\end{equation*}
where $\alpha_3(\varepsilon)\rightarrow0$ as $\varepsilon\rightarrow
0.$
Since this inequality holds for every $f\in {\frak F_{Q}},$ the
family ${\frak F_{Q}}$ is equicontinuous.
\end{proof}

\subsection{} Let $Q:D\rightarrow [0,\infty]$ be a Lebesgue measurable function. Denote by $q_{x_0}$ the mean value of $Q(x)$ over the sphere $|x-x_0|=r$, that means,
\begin{equation}\label{eq32*}
q_{x_0}(r):=\frac{1}{\omega_{n-1}r^{n-1}}\int\limits_{|x-x_0|=r}Q(x)\,d{\mathcal
H}^{n-1}\,,
\end{equation}
where $\omega_{n-1}$ denotes the area of the unit sphere in $\mathbb R^n$.

The proof of the main results is based on the following statement (for $p=n$ we refer to \cite{Sev10}).

\medskip
\begin{lemma}\label{lem4A} 
{\em Let a function $Q:D\rightarrow [0,\infty]$ be Lebesgue measurable in a domain $D\subset \mathbb R^n,$ $n\ge 2,$ and $x_0\in
D$. Assume that either of the following conditions holds

\noindent (a) $Q\in {\rm FMO}(x_0),$

\noindent (b) $q_{x_0}(r)\,=\,O\left(\left[\log{\frac1r}\right]^{n-1}\right)$ as
$r\rightarrow 0,$

\noindent (c) for some small
$\delta_0=\delta_0(x_0)>0$ we have the relations
\begin{equation}\label{eq5***}
\int\limits_{\delta}^{\delta_0}\frac{dt}{tq_{x_0}^{\frac{1}{n-1}}(t)}<\infty,\qquad 0<\delta<\delta_0,
\end{equation}
and
\begin{equation}\label{eq5**}
\int\limits_{0}^{\delta_0}\frac{dt}{tq_{x_0}^{\frac{1}{n-1}}(t)}=\infty\,.
\end{equation}
Then  there exist a number $\varepsilon_0\in(0,1)$ and a function $\psi(t)\ge 0$ such that the inequalities (\ref{eq27*!}) and (\ref{eq1A})
of Lemma \ref{lem3*} hold at the point $x_0$ for any $p$, $0<p\le n$. }
\end{lemma}

\medskip
\begin{proof} We shall show in the proof that in the case, when the assumption (c) holds, one can choose $\varepsilon_0=\delta_0$.
Without loss of generality, assume that $x_0=0.$

We start with the assumption $Q\in {\rm FMO}$. Due to
\cite[Corollary~6.3, Ch.~6]{MRSY09}, the condition $Q\in
{\rm FMO(0)}$ implies that for some small $\varepsilon<\varepsilon_0$
\begin{equation}\label{eq31*}
\int\limits_{\varepsilon<|x|<\varepsilon_0}Q(x)\cdot\psi^p(|x|)
\ dm(x)\,=\,O
\left(\log\log \frac{1}{\varepsilon}\right)\,,\quad \varepsilon\to 0\,,
\end{equation}
with $\psi(t)=\left(t\,\log{\frac1t}\right)^{-n/p}>0.$
Note that the quantity $I(\varepsilon, \varepsilon_0)$ defined in Lemma \ref{lem3*} is estimated by
\begin{equation}\label{eqlogest}
I(\varepsilon,
\varepsilon_0)=\int\limits_{\varepsilon}^{\varepsilon_0}\psi(t) dt
>\log{\frac{\log{\frac{1}
{\varepsilon}}}{\log{\frac{1}{\varepsilon_0}}}}.
\end{equation}
Thus the estimate (\ref{eq31*}) yields
\begin{equation*}
\frac{1}{I^p(\varepsilon,
\varepsilon_0)}\int\limits_{\varepsilon<|x|<\varepsilon_0}
Q(x)\cdot\psi^p(|x|)  \ dm(x)\le C\left(\log\log\frac{1}{\varepsilon}\right)^{1-p}\rightarrow 0, \quad \varepsilon\to 0\,,
\end{equation*}
which completes the proof for the case (a).

Consider now the case (b), i.e.,
$q_{x_0}(r)\,=\,O\left(\left[\log{\frac1r}\right]^{n-1}\right)$ as
$r\rightarrow 0$ and fix arbitrary $\varepsilon_0$ providing $\varepsilon_0<\min\left\{{\rm
dist}\,\left(0,\partial D\right),1\right\}.$ Letting
$\psi(t)=\left(t\,\log{\frac1t}\right)^{-n/p},$ one can establish that
\begin{equation*}
\begin{split}
\int\limits_{\varepsilon<|x|<\varepsilon_0} Q(x)\cdot\psi^p(|x|)
 \ dm(x)\,&= \int\limits_{\varepsilon\,<\,|x|\,<\,\varepsilon_0}
\frac{Q(x)dm(x)}{\left(|x|\log{\frac{1}{|x|}}\right)^n}\,
\\&
=\,\int\limits_{\varepsilon}^{\varepsilon_0}
\left(\,\int\limits_{|x|\,=\,r}\frac{Q(x)}{\left(|
x|\log{\frac{1}{|x|}}\right)^n}\, d{\mathcal
H}^{n-1}\,\right)\,dr
\\&\le
C\omega_{n-1}\,\int\limits_{\varepsilon}^{\varepsilon_0}\,
\frac{dr}{r\log{\frac1r}}=C\omega_{n-1}\log{\frac{\log{\frac{1}
{\varepsilon}}}{\log{\frac{1}{\varepsilon_0}}}}\,\le\,C\omega_
{n-1}\cdot I(\varepsilon, \varepsilon_0)\,,
\end{split}
\end{equation*}
where $I(\varepsilon,
\varepsilon_0)$ is the same as above and tends to $\infty$ as $\varepsilon\to 0.$ This implies
\begin{equation*}
\frac{1}{I^p(\varepsilon,
\varepsilon_0)}\int\limits_{\varepsilon<|x|<\varepsilon_0}
Q(x)\cdot\psi^p(|x|)
 \ dm(x)\,\rightarrow 0\quad \text{as}\quad \varepsilon\rightarrow 0,
\end{equation*}
and thereby the assertion of the lemma.

In the case (c), letting
$\varepsilon_0=\delta_0$ and $\varepsilon=\delta<\varepsilon_0,$
consider the function
$I(\varepsilon, \varepsilon_0)=\int\limits
_{\varepsilon}^{\varepsilon_0}\psi(t)dt$ with
\begin{equation}\label{eq44**}
\psi(t)\quad=\quad \left \{\begin{array}{rr}
\left(1/[tq^{\frac{1}{n-1}}_{0}(t)]\right)^{n/p}\ , & \ t\in
(\varepsilon, \varepsilon_0)\ ,
\\ 0\ ,  &  \ t\notin (\varepsilon,
\varepsilon_0)\ ,
\end{array} \right.
\end{equation}
and $q_0(r)=q_{{x_0}}(r),$ $x_0=0$ (as usual,
$a/\infty = 0$ for $a\ne\infty,$ $a/0=\infty $ for $a>0$ and
$0\cdot\infty =0,$ cf. \cite[Ch.~I]{Sak64}).
It follows from (\ref{eq5***}) that $I(\varepsilon, \varepsilon_0)<\infty$.
For $p\le n,$ the H\"older inequality yields
\begin{equation}\label{eq3B***}
\int\limits_{\varepsilon}^{\varepsilon_0}\frac{dt}{tq_{0}^{\frac{1}{n-1}}(t)}\le
\left(\int\limits_{\varepsilon}^{\varepsilon_0}\left(
\frac{1}{tq_{0}^{\frac{1}{n-1}}(t)} \right)^{n/p}
dt\right)^{p/n}\cdot
\left(\varepsilon_0-\varepsilon\right)^{(n-p)/n}\,.
\end{equation}
Thus, the equality (\ref{eq5**}) implies that $I(\varepsilon,
\varepsilon_0)>0$ for some $\varepsilon_1\in (0, \varepsilon_0)$ and any $\varepsilon\in(0, \varepsilon_1)$.
In addition, $\psi$ satisfies (\ref{eq27*!}). Actually, applying Fubini's theorem immediately derives
\begin{equation}\label{eq26}
\int\limits_{\varepsilon<|x|<\varepsilon_0} Q(x)\cdot\psi^p(|x|)\
dm(x)\ =\ \omega_{n-1}\cdot
\int\limits_{\varepsilon}^{\varepsilon_0}\frac{dt}{tq_{0}^{\frac{1}{n-1}}(t)};
\end{equation}
hence by (\ref{eq5**}) and (\ref{eq3B***}),
\begin{equation*}
\int\limits_{\varepsilon}^{\varepsilon_0}\frac{dt}{tq_{0}^{\frac{1}{n-1}}(t)}
=o\left(I^p(\varepsilon,\varepsilon_0)\right).
\end{equation*}
Lemma \ref{lem4A} is proved completely.
\end{proof}

\medskip
In the next section we shall apply the following useful statement.

\medskip
\begin{lemma}\label{pr2} 
{\sl Let $f:D\rightarrow {\mathbb R}^n$ be a discrete open $(p,
Q)$-mapping at a point $x_0$ and $p\in (n-1, n).$ Then (\ref{eq5***})
holds for every $0<\delta<\delta_0$ and $\delta_0\in (0, {\rm
dist\,}(x_0,
\partial D)).$}
\end{lemma}

\medskip
\begin{proof}
Consider a condenser $E=\left(B(x_0, r_2),
\overline{B(x_0, r_1)}\right),$ $0<r_1<r_2< {\rm dist} \,
(x_0,\partial D),$ and denote
\begin{equation*}
\widetilde{I}:=\widetilde{I}(x_0,r_1,r_2)=\int\limits_{r_1}^{r_2}\
\frac{dr}{r^{\frac{n-1}{p-1}}q_{x_0}^{\frac{1}{p-1}}(r)}\,.
\end{equation*}
We show that this integral is convergent.

Due to \cite[Lemma~1]{SS12}, $p$-capacity of the condenser $f(E)=\left(f\left(B(x_0, r_2)\right), f\left(\overline{B(x_0,
r_1)}\right)\right)$ is estimated by
\begin{equation}\label{eq2A}
{\rm cap}_p\, f(E)\,\le\frac{\omega_{n-1}}{\widetilde{I}^{p-1}}.
\end{equation}
Now observe that $\widetilde{I}<\infty.$ Indeed, were $\widetilde{I}=\infty$ one obtains from (\ref{eq2A}) and from
\cite[Theorem~1.15, Ch.~VII]{Ric93} that ${\rm dim\,}f(E)=0,$ thus ${\rm Int} f(E)=\varnothing.$ But this
contradicts to the openness of $f.$

By the H\"{o}lder inequality
\begin{equation*}
\int\limits_{r_1}^{r_2}\frac{dt}{tq^{1/(n-1)}_{x_0}(t)}\le \widetilde{I}^{\frac{p-1}{n-1}}\cdot
(r_2-r_1)^{\frac{n-p}{n-1}}\,<\,\infty,
\end{equation*}
which completes the proof.
\end{proof}

\section{Main results}

Based on Lemmas \ref{lem4}--\ref{pr2} from the previous section, we are now able to prove the main results of the paper starting with the proof of Theorem \ref{th1}.

\medskip
\begin{proof}[Proof of Theorem \ref{th1}]
The proof follows from Lemmas \ref{lem1} and \ref{lem4A} using the admissible function
\begin{equation*}
\psi(t)=\frac{1}{\left(t\,\log{\frac1t}\right)^{n/p}}.
\end{equation*}
\end{proof}

\medskip
We also have

\begin{theo}\label{th2} 
{\em Let $Q:D\rightarrow(0, \infty]$ be a Lebesgue measurable
function, $x_0\in D$, $n-1<p<n,$ and let $\frak{F}_{p, Q}(b)$ be a
family of all open discrete $(p, Q)$-mappings $f:D\rightarrow {\Bbb
R}^n$ at $b.$ Assume that
$q_{x_0}(r)=O\left(\left[\log{\frac1r}\right]^{n-1}\right)$ as
$r\rightarrow 0,$ or more generally, that for some $\delta_0,$
$0<\delta_0<{\rm dist\,}(x_0,\partial D),$ and all $\delta\in
(0,\delta_0)$ the relation (\ref{eq5**}) holds. Then the family
$\frak{F}_{p, Q}(b)$ is equicontinuous at $b.$}
\end{theo}

\medskip
\begin{proof}
If $q_{x_0}(r)\,=\,O\left(\left[\log{\frac1r}\right]^{n-1}\right),$ one can take
\begin{equation}\label{eqadmf}
\psi(t)=\frac{1}{\left(t\,\log{\frac1t}\right)^{n/p}},
\end{equation}
and then the conclusion of the theorem follows from Lemmas \ref{lem1} and
\ref{lem4A}.

If instead, we have more general assumptions on $Q$, the same Lemmas
\ref{lem1} and \ref{lem4A} and (\ref{eqadmf}) also are applicable, because the inequality (\ref{eq5***}) holds by
Lemma \ref{pr2} and (\ref{eq5**}) holds by assumption.
\end{proof}

It is well known that ${\rm FMO}\not\subset L^p$ for any $p>1$ and $L^p\not\subset {\rm FMO}$ for any sufficiently large parameter $p$; see the examples in \cite{MRSY09}. It would be interesting to find a relation between the degree of local integrability of the function $Q$ and equicontinuity for the ring $(p,Q)$--mappings. Note that in the case $p=n,$ any positive degree of local integrability of $Q$ does not imply neither equicontinuity nor normality of the family of such mappings; cf. \cite{MRSY09}.

\medskip
For $p\ne n$ we have the following

\begin{theo}\label{th3} 
{\em Let $Q:D\rightarrow (0, \infty]$ be a Lebesgue measurable
function, ${\frak F_{Q}}$ be a family of open discrete ring $(p,
Q)$-mappings $f:D\rightarrow{\Bbb R}^n$ at $b\in D,$ $p\in (n-1,
n).$ Assume that that $Q\in L^s(D),$ with $s\ge\frac{n}{n-p}.$ Then
${\frak F_{Q}}$ is equicontinuous at $b.$}
\end{theo}

\medskip
\begin{proof}
Pick $\varepsilon_0<{\rm dist}\,(x_0,\partial D)$ and $\psi(t)=1/t$ as in Lemma \ref{lem3*}.
This function satisfies (\ref{eq1A}), thus it remains to establish the estimate (\ref{eq27*!}).

The H\"older inequality yields
\begin{equation*}
\int\limits_{\varepsilon<|x-b|<\varepsilon_0}\frac{Q(x)}{|x-b|^p} \
dm(x)\leq \left(\int\limits_{\varepsilon<|x-b|<\varepsilon_0}\frac{1}{|x-b|^{pq}}
\ dm(x) \right)^{\frac{1}{q}}\,\left(\int\limits_{D} Q^{q^{\prime}}(x)\
dm(x)\right)^{\frac{1}{q^{\prime}}}.
\end{equation*}
where  $1/q+1/q^{\prime}=1$. Note that the first integral in the right-hand side can be calculated explicitly. And the calculation is split into two cases.

Taking $q^{\prime}=\frac{n}{n-p}$ (and hence $q=\frac{n}{p},$) and using Fubini's theorem, one gets
\begin{equation*}
\int\limits_{\varepsilon<|x-b|<\varepsilon_0}\frac{1}{|x-b|^{pq}}
\ dm(x)=\omega_{n-1}\int\limits_{\varepsilon}^{\varepsilon_0}
\frac{dt}{t}=\omega_{n-1}\log\frac{\varepsilon_0}{\varepsilon}\,.
\end{equation*}
Then the first integral in Lemma \ref{lem3*} is estimated by
\begin{equation*}
\frac{1}{I^p(\varepsilon, \varepsilon_0)}\int\limits_{\varepsilon<|x-b|<\varepsilon_0}\frac{Q(x)}{|x-b|^p}
\ dm(x)\le
\omega^{\frac{p}{n}}_{n-1}\|Q\|_{\frac{n}{n-p}}\left(\log\frac{\varepsilon_0}{\varepsilon}\right)
^{-p+\frac{p}{n}}\,\to 0\,,
\end{equation*}
as $\varepsilon\rightarrow 0,$ which is equivalent to (\ref{eq27*!}).

\medskip
Let now $q^{\prime}>\frac{n}{n-p}$ (and $q=\frac{q^{\prime}}{q^{\prime}-1}$). In this case,
\begin{equation*}
\int\limits_{\varepsilon<|x-b|<\varepsilon_0}\frac{1}{|x-b|^{pq}}
\ dm(x) = \omega_{n-1}\int\limits_{\varepsilon}^{\varepsilon_0}
t^{n-\frac{pq^{\prime}}{q^{\prime}-1}-1}dt\le
\omega_{n-1}\int\limits_{0}^{\varepsilon_0} t^{n-\frac{pq^{\prime}}{q^{\prime}-1}-1}dt
=\frac{\omega_{n-1}}{n-\frac{pq^{\prime}}{q^{\prime}-1}}\varepsilon^{n-\frac{pq^{\prime}}{q^{\prime}-1}}_0,
\end{equation*}
and
\begin{equation*}
\frac{1}{I^p(\varepsilon, \varepsilon_0)}
\int\limits_{\varepsilon<|x-b|<\varepsilon_0}\frac{Q(x)}{|x-b|^p} \
dm(x)\le \|Q\|_{q^{\prime}}\left(\frac{\omega_{n-1}}{n-\frac{pq^{\prime}}{q^{\prime}-1}}
\varepsilon^{n-\frac{pq^{\prime}}{q^{\prime}-1}}_0\right)^{\frac{1}{q}}\left(\log\frac{\varepsilon_0}{\varepsilon}\right)^{-p}\,,
\end{equation*}
which implies the assumptions (\ref{eq27*!})--(\ref{eq1A}) of Lemma \ref{lem3*}.
Now the assertion of the theorem follows from Lemma \ref{lem1}.
\end{proof}

The following theorem provides another sufficient condition for equicontinuity of  discrete open $(p,Q)$-mappings.

\medskip
\begin{theo}\label{th4}  
{\em Let $Q:D\rightarrow (0, \infty]$ be a Lebesgue measurable
function, ${\frak F_{Q}}$ be a family of open discrete ring $(p,
Q)$-mappings $f:D\rightarrow{\Bbb R}^n$ at $b\in D,$ $p\in (n-1,
n).$ Assume that there exists $0<\varepsilon_0<{\rm dist\,}\left(b,
\partial D\right)$ such that
\begin{equation*}
\int\limits_{0}^{\varepsilon_0}\ \frac{dr}{r^{\frac{n-1}{p-1}}
q_{b}^{\frac{1}{p-1}}(r)}=\infty\,,
\end{equation*}
where $q_{b}(r)$ is defined by (\ref{eq32*}). Then ${\frak F_{Q}}$
is equicontinuous at $b.$}
\end{theo}

\medskip
\begin{proof}
Applying again (\ref{eq2A}), one derives similar to the proof of Lemma \ref{pr2}, that
\begin{equation*}
\int\limits_{\varepsilon}^{\varepsilon_0}\
\frac{dr}{r^{\frac{n-1}{p-1}} q_{b}^{\frac{1}{p-1}}(r)}<\infty,
\qquad 0<\varepsilon<\varepsilon_0.
\end{equation*}

Pick
\begin{equation*}   
\psi(t)= \left \{\begin{array}{rr}
1/[t^{\frac{n-1}{p-1}}q_{b}^{\frac{1}{p-1}}(t)]\ , & \ t\in
(\varepsilon,\varepsilon_0)\ ,
\\ 0\ ,  &  \ t\notin (\varepsilon,\varepsilon_0)\ ,
\end{array} \right.
\end{equation*}
and consider the corresponding integral
(\ref{eq1A}),
which satisfies
$0<I(\varepsilon, \varepsilon_0)<\infty$ for some $\varepsilon_1\in
(0, \varepsilon_0)$ and all $\varepsilon\in (0, \varepsilon_1).$
Using the Fubini theorem, one gets
\begin{equation*}
\int\limits_{\varepsilon<|x-b|<\varepsilon_0}Q(x)\cdot\psi^p(|x-b|)
\ dm(x)=\omega_{n-1}\,\int\limits_{\varepsilon}^{\varepsilon_0}\
\frac{dr}{r^{\frac{n-1}{p-1}} q_{b}^{\frac{1}{p-1}}(r)}\,,
\end{equation*}
and now the assertion of Theorem \ref{th4} follows from Lemma \ref{lem1}.
\end{proof}

\medskip

\begin{corol}\label{cor1*}
{\em The assertion of Theorem \ref{th4} holds if $q_b(t)\le
ct^{p-n}$ for some constant $c>0$ and almost all $t\in (0,
\varepsilon_0),$ providing that $\varepsilon_0<{\rm dist\,}(b, \partial D)$ is sufficiently small.}
\end{corol}

\section{Distortion estimates for bounded ring $(p, Q)$-mappings}

In this section we establish somewhat explicit
estimates.

\medskip
\begin{lemma}\label{lem6} 
{\em
Let $f:D\rightarrow{\mathbb R}^n,$ $n\ge 2$ be a discrete open ring
$(p, Q)$-map\-ping at a point $x_0\in D,$ $D^{\,\prime}:=f(D)\subset
B(0, r).$ Suppose that there exist numbers $q \le p,$ $\varepsilon_0
\in (0,\, {\rm dist}\,(x_0,
\partial D)),$ $\varepsilon_0^{\,\prime}\in (0, \varepsilon_0)$ and nonnegative
Lebesgue measurable functions $\psi_{\varepsilon}:(\varepsilon,
\varepsilon_0)\rightarrow [0, \infty],$ $\varepsilon\in\left(0,
\varepsilon_0^{\,\prime}\right)$ such that
\begin{equation} \label{eq3.7A}
\int\limits_{\varepsilon<|x-x_0|<\varepsilon_0}Q(x)\cdot\psi_{\varepsilon}^p(|x-x_0|)
\ dm(x)\le K\cdot I^q(\varepsilon, \varepsilon_0)\qquad \forall\,\,
\varepsilon\in(0,\varepsilon_0^{\,\prime})\,,
\end{equation}
where $I(\varepsilon, \varepsilon_0)$ is defined by
(\ref{eq3}). Then
\begin{equation*}
|f(x)-f(x_0)|\le C r^{\frac{(1-n+p)n}{p}}K^{\frac{n-1}{p}}
I^{\frac{(q-p)(n-1)}{p}}(|x-x_0|, \varepsilon_0)\,.
\end{equation*}
for every $x\in B(x_0,{\varepsilon_0}^{\,\prime}),$ where $C$ is
a constant depending only on $n$ and $p.$}
\end{lemma}

\medskip
\begin{proof}
Letting in (\ref{eq3B}) $\Phi(\varepsilon,
\varepsilon_0):=K\cdot I^q(\varepsilon, \varepsilon_0),$ one obtains from (\ref{eq3.7A}),
\begin{equation}\label{eq3a}
{\rm cap}_p\,f(E)\le K\cdot I^{q-p}\left(\varepsilon,
\varepsilon_0\right)\,.
\end{equation}

Since $f(A)\subset B(0, r),$ the bound (\ref{2.5}) yields
\begin{equation}\label{eq17}
{\rm cap}_p\,f(E)\ge \left(c_1\frac{\left(d(f(C))\right)^p}
{\left(m(f(A))\right)^{1-n+p}}\right)^{\frac{1}{n-1}}\ge
\left(c_1\frac{\left(d(f(C))\right)^p}
{\left(\Omega_nr^{n})\right)^{1-n+p}}\right)^{\frac{1}{n-1}}\,.
\end{equation}
It follows from (\ref{eq17}) and  (\ref{eq3a}) that
\begin{equation}\label{eq20}
d\left(f(C)\right)\le\left(\frac{1}{c_1}\right)^{1/p}
\Omega_n^{\frac{1-n+p}{p}}r^{\frac{(1-n+p)n}{p}}K^{\frac{n-1}{p}}
I^{\frac{(q-p)(n-1)}{p}}(\varepsilon, \varepsilon_0)\,.
\end{equation}
Now let $x\in D$ be such that $|x-x_0|=\varepsilon,$ $0<
\varepsilon<\varepsilon_0^{\,\prime}.$ Then,
$x\in\overline{B\left(x_0,\,\varepsilon\right)}$ and
$f(x)\in\,f\left(\overline{B\left(x_0,\,\varepsilon\right)}\right)=f(C),$
and from (\ref{eq20}) we get the estimate
\begin{equation}\label{eq21}
|f(x)-f(x_0)|\le\left(\frac{1}{c_1}\right)^{1/p}
\Omega_n^{\frac{1-n+p}{p}}r^{\frac{(1-n+p)n}{p}}K^{\frac{n-1}{p}}
I^{\frac{(q-p)(n-1)}{p}}(|x-x_0|, \varepsilon_0)\,.
\end{equation}
Since
$\varepsilon\in\left(0,\varepsilon_0^{\,\prime}\right)$ is chosen arbitrary, the
relation (\ref{eq21}) holds in the whole ball $B(x_0,
\varepsilon_0^{\,\prime}).$
\end{proof}

\medskip
\begin{theo}\label{th6} 
{\em Let $f:D\rightarrow{\mathbb R}^n,$ $n\ge 2$ be an open discrete ring
$(p, Q)$-map\-ping at a point $x_0\in D,$ and $D^{\,\prime}=f(D)\subset
B(0, r).$ If $Q\in {\rm FMO}(x_0),$ then 
\begin{equation*}
|f(x)-f(x_0)|\le C(p, n, x_0, r)
\left[\log\log\frac{1}{|x-x_0|}\right]^{\frac{(1-p)(n-1)}{p}}
\end{equation*}
for some constant $C(p, n, x_0, r)$ depending only on $p,$ $n,$
$x_0$ and $r,$ and every $x\in B(x_0,{\varepsilon_0}^{\,\prime})$
with $0<{\varepsilon_0}^{\,\prime}<{\rm dist}\,(x_0, \partial D).$
}
\end{theo}

\medskip
\begin{proof}
This follows from Lemma \ref{lem6}, letting
$\psi(t)=\left(t\,\log{\frac1t}\right)^{-n/p}>0.$ If
$\varepsilon_0>0$ is sufficiently small, we obtain
from (\ref{eqlogest}) that
\begin{equation*}
I(\varepsilon,
\varepsilon_0)=\int\limits_{\varepsilon}^{\varepsilon_0}\psi(t)dt
>C\cdot\log\log{\frac{1}
{\varepsilon}}
\end{equation*}
for some constant $C>0$ and $\varepsilon<\varepsilon_0$.
Using (\ref{eq31*}), one obtains that $f$ obeys the inequality
(\ref{eq3.7A}) with $q=1$. Now the conclusion of the theorem follows from
Lemma \ref{lem6}.
\end{proof}

\medskip
\begin{theo}\label{th5} 
{\em
Let $f:D\rightarrow{\Bbb R}^n,$ $n\ge 2$ be an open discrete ring
$(p, Q)$-map\-ping at a point $x_0\in D,$ $D^{\,\prime}:=f(D)\subset
B(0, r).$ If (\ref{eq5**}) holds for some $\delta_0\in (0, {\rm
dist}\, (x_0, \partial D)),$ then 
\begin{equation*}
|f(x)-f(x_0)|\le C(p, n, x_0, r)
\left(\int\limits_{|x-x_0|}^{\delta_0}\frac{dt}{tq^{1/(n-1)}_{x_0}(t)}\right)
^{\frac{-(n-1)^2}{n}}
\end{equation*}
for some constant $C(p, n, x_0, r)$ depending only on $p,$ $n,$
$x_0$ and $r,$ and every $x\in B(x_0,{\varepsilon_0}^{\,\prime});$
here $0<{\varepsilon_0}^{\,\prime}<{\rm dist}\,(x_0, \partial D).$}
\end{theo}

\medskip
\begin{proof}
This statement follows from Lemma \ref{lem6}, taking
$\psi(t)$ as in (\ref{eq44**}). Now (\ref{eq3.7A})
holds for $q=p/n$ (in view of (\ref{eq3B***}) and (\ref{eq26})), and the assertion of Theorem \ref{th5} again follows from Lemma \ref{lem6}.
\end{proof}

\medskip
{\small \leftline{\textbf{Anatoly Golberg}} \em{
\leftline{Department of Applied Mathematics,} \leftline{Holon
Institute of Technology,} \leftline{52 Golomb St., P.O.B. 305,}
\leftline{Holon 5810201, ISRAEL} \leftline{Fax: +972-3-5026615}
\leftline{e-mail: golberga@hit.ac.il}}}

\medskip

{\small \leftline{\textbf{Ruslan Salimov}}\em{ \leftline{Institute of
Applied Mathematics and Mechanics,}\leftline {National Academy of
Sciences of Ukraine,} \leftline{74 Roze Luxemburg St.,}
\leftline{Donetsk 83114, UKRAINE}\leftline{e-mail:
ruslan623@yandex.ru}}}

\medskip

{\small \leftline{\textbf{Evgeny Sevost'yanov}}\em{ \leftline{Institute of
Applied Mathematics and Mechanics,}\leftline {National Academy of
Sciences of Ukraine,} \leftline{74 Roze Luxemburg St.,}
\leftline{Donetsk 83114, UKRAINE}\leftline{e-mail:
brusin2006@rambler.ru}}}

\end{document}